\documentclass[11pt,leqno]{amsart}

\usepackage[top=2.5 cm,bottom=2 cm,left=2.5 cm,right=2 cm]{geometry}
\usepackage[utf8]{inputenc}

\usepackage{esint}
\usepackage{graphicx}
\usepackage{enumerate}
\newtheorem{definition}{Definition}[section]
\newtheorem{theorem}{Theorem}[section]
\newtheorem{lemma}{Lemma}[section]
 
\newtheorem{remark}{Remark}[section]
\newtheorem*{maintheorem*}{Main Theorem}
\allowdisplaybreaks
\numberwithin{equation}{section}

\renewcommand{\i}{\ifmmode\mathit{\mathchar"7010 }\else\char"10 \fi}
\renewcommand{\j}{\ifmmode\mathit{\mathchar"7011 }\else\char"11 \fi}
\newcommand{\R}{\mathbb{R}}
\newcommand{\N}{\mathbb{N}}

\newcommand{\sgn}[1]{\mathrm{sign}\left(#1\right)}

\newcommand{\norm}[1]{\left\|#1\right\|}

\newcommand{\dv}[1]{\mathrm{div}\left(#1\right)}

\newcommand{\pt}{\partial_t}

\newcommand{\px}{\partial_x }

\newcommand{\p}{\partial}

\newcommand{\vfi}{\varphi}
\newcommand{\eps}{\varepsilon}

\newcommand{\Hneg}{H_{loc}^{-1}}
\newcommand{\Lip}{{\rm Lip}}

\def\begi{\begin{itemize}}
\def\endi{\end{itemize}}
\def\bega{\begin{array}}
\def\enda{\end{array}}

\def\bel{\begin{equation}\label}
\def\eeq{\end{equation}}

\newcommand{\ue}{u_\eps}
\newcommand{\ume}{u_{\mu,\eps}}

\newcommand{\uek}{u_{\eps_{k}}}

\newcommand{\CL}{\mathcal{L}}

\newenvironment{Assumptions}% Definition of assumptions
{%

\begin{enumerate}}%
{\end{enumerate}}

\newenvironment{Assumptionss}% Definition of assumptions
{%

\begin{enumerate}}%
{\end{enumerate}}

{%

\begin{enumerate}}%
{\end{enumerate}}

\begin{document}

\title[Vanishing Viscosity Limits]{Vanishing Viscosity Limits of Scalar Equations\\ with Degenerate Diffusivity}

\author[G. M. Coclite]{G. M. Coclite}
\address[Giuseppe Maria Coclite]
{\newline Department of Mechanics, Mathematics and Management, Polytechnic University of Bari, via E. Orabona 4, 70125 Bari,   Italy}
\email[]{giuseppemaria.coclite@poliba.it}
\urladdr{http://www.dm.uniba.it/Members/coclitegm/}

\author[A. Corli]{A. Corli}
\address[Andrea Corli]
{\newline Department of Mathematics and Computer Science, University of Ferrara, Via Machiavelli 30, I--44121 Ferrara, Italy.}
\email[]{andrea.corli@unife.it}
\urladdr{http://docente.unife.it/andrea.corli}

\author[L. di Ruvo]{L. di Ruvo}
\address[Lorenzo di Ruvo]
{\newline Department of Mathematics, University of Bari, via E. Orabona 4, 70125 Bari,   Italy}
\email[]{lorenzo.diruvo77@gmail.it}

\date{\today}



\thanks{The authors are members of the Gruppo Nazionale per l'Analisi Matematica, la Probabilit\`a e le loro Applicazioni (GNAMPA) of the Istituto Nazionale di Alta Matematica (INdAM).\\
This work was initiated while G.~M.~Coclite visited Department of Mathematics and Computer Science at the University of Ferrara.
He is grateful for Department's financial support and excellent working conditions.}

\begin{abstract} We consider a scalar, possibly degenerate parabolic equation with a source term, in several space dimensions. For initial data with bounded variation we prove the existence of solutions to the initial-value problem. Then we show that these solutions converge, in the vanishing-viscosity limit, to the Kruzhkov entropy solution of the corresponding hyperbolic equation. The proof exploits the $H$-measure compactness in several space dimensions.
\end{abstract}

\maketitle

%\tableofcontents

\section{Introduction}
\label{sec:intro}

In this paper we consider the scalar equation
\begin{equation}
\label{eq:sd-sola}
\pt u+\dv{f(u)}=\eps\Delta A(u)+g(t,x,u),
\end{equation}
for $t>0$ and $x\in\R^N$. The functions $f$, $A$ and $g$ are supposed to be smooth and $\eps>0$ is given. In the case of smooth solutions we also write the diffusion term as $\eps\Delta A(u) = \eps\dv{a(u)\nabla u}$, for $a=A'$. The term $\eps a$ (and also $a$, with a slight abuse of terminology) is called the diffusivity.

Reaction-diffusion-advection equations as \eqref{eq:sd-sola} are known to model a great variety of physical and biological phenomena. In particular, several examples can be found in \cite{GK,Kalashnikov,Murray,Vazquez} in the case the diffusivity $a\ge0$ vanishes at some point or even in a set with positive measure. 

Another example, which motivated this research, comes from the modeling of collective movements, where $N=1$. In the case of traffic flows, the conservation law $\pt u + \dv{uv(u)}=0$ takes the name of Lighthill-Whitham-Richards (LWR) equation \cite{Lighthill-Whitham, Richards}; in this case the function $u$ is valued in $[0,1]$ and plays the role of a normalized vehicle density, while $v$ represents the velocity. The term $\eps\Delta A(u)=\eps A(u)_{xx}$ was then introduced to avoid the occurrence of discontinuities in the solution \cite{Lighthill-Whitham,Payne}, while the term $g$ is motivated by the presence of entries or exits in the road \cite{Bagnerini-Colombo-Corli}. The presence of the diffusion is also motivated in \cite{Nelson_2000} by taking into consideration an anticipation distance $\eps$ (and possibly a relaxation time); in the case of crowd dynamics, $\eps$ is interpreted as the characteristic depth of the visual field of pedestrians \cite{BTTV}. In both cases, equation $\eqref{eq:sd-sola}$ is formally deduced from the LWR equation by a first-order Taylor expansion of the density $u$ with respect to $\eps$. The interesting feature is that in this expansion the diffusivity is computed as $a(u)=-uv'(u)$ and in particular {\em vanishes at $u=0$}. Several related models of diffusion are provided in \cite{Bellomo-Delitala-Coscia}. 

Two issues gave rise to the present paper. First, the existence and uniqueness of solutions to the initial-value problem for \eqref{eq:sd-sola} with initial datum
\begin{equation}\label{e:id}
u(0,x)=u_0(x),\quad x\in\R^N,
\end{equation}
for a {\em degenerate} diffusivity $a$; second, the {\em convergence} for $\eps\to0$ of the solutions $u_\eps$ of \eqref{eq:sd-sola}, \eqref{e:id} to the entropy solution of the balance law
\begin{equation}
\label{eq:sd-solac}
\pt u+\dv{f(u)}=g(t,x,u),
\end{equation}
having the same initial datum \eqref{e:id}. The aim of the program is to provide a rigorous justification to the approximation procedure in \cite{BTTV, Nelson_2000}.

\smallskip

Several papers dealt with both issues; we now briefly quote only those that enter in the framework of this paper, referring the reader to their references for further information. 

We first discuss the papers dealing with the general case $N\ge1$. In the fundamental and very technical paper \cite{Volpert-Hudjaev} the smooth functions $f$, $A$ and $g$ are allowed to depend also on $t,x$; moreover, only the assumption $a\ge0$ is required. There, the authors show the existence of $BV$ (bounded variation) entropy solutions by the vanishing viscosity method. They also obtain some results about uniqueness; in the case of one spatial dimension a complete uniqueness result in the $BV$ class is provided in \cite{Wu-Yin}, see also \cite{Volpert2000}. We refer to \cite{Kalashnikov} for related results. If $f=f(x,u)$ satisfies suitable assumptions but its spatial dependence is merely continuous (non-Lipschitz) we refer to \cite{Karlsen-Risebro}, where some fundamental results of \cite{Carrillo} are exploited. 
%The initial-boundary value problem is considered, in the case the diffusion is not degenerate, in the recent paper \cite{Mondal-Ganesh-Baskar}, by exploiting suitable $BV$ estimates, and in \cite{Mondal-Ganesh}, via the compensated compactness \footnote{E allora bisogna giustificare perche' loro non hanno bisogno di Panov ma della teoria classica della CC! Nota inoltre che non hanno NESSUNA ipotesi di convessita' o genuina non linearita' su $f$! {\bf Giuseppe: Gli articoli 26, 27 e 28 sono scritti veramente maluccio. Le ipotesi non sono affatto chiare. Sembra che si ridimostrino la compattezza compensata. Io non li citerei nemmeno. Aspettiamo che passino per le mani di un referee e vediamo cosa ne esce fuori.}}; see also \cite{Mondal-Ganesh3} for a refinement of these results. 

The case $N=1$ offers much more results, in particular because the theory of compensated compactness was available in that case since the late seventies. In the case $g=g(t,x)$, if $f,A$ are barely continuous, $A$ weakly increasing and $u_0, g$ are non-smooth, then existence and uniqueness of solutions is obtained in \cite{Benilan-Toure_CRAS, Benilan-Toure_IHP}  by the nonlinear semigroup theory; the limit $\eps\to0$ is also studied.  If $g=0$, the case when the diffusion is $\eps\left(b(u_x)\right)_x$ is considered in \cite{Marcati-Natalini} under suitable assumptions on both $f$ and $b$. If $a$ vanishes identically on an interval $[a,b]$, $a<b$, and it is strictly positive elsewhere, see \cite{Marcati}. The convergence of absolutely continuous solutions of \eqref{eq:sd-sola} to a solution of \eqref{eq:sd-solac} is showed in \cite{Marcati_1988}. 

Other results, again in the case $N=1$, are the following. In \cite{Jager-Lu}, under the assumptions that $a$ vanishes at most in zero-measure set and $g=0$, the authors show that if $f$ and $A$ are smooth and $\norm{u_0}_{L^1}\le M$, then there is a unique weak continuous solution satisfying the regularity estimate $\norm{A(u)}_{L^1}\le M$. Other results are given in \cite{Gilding} in the case that $u_0,u$ are continuous. The case when $f=f(x,u)$ depends in a discontinuous way on $x$ is considered in \cite{Karlsen-Risebro-Towers_2003}. Existence and uniqueness of $BV$ solutions to an initial-boundary value problem was proved in \cite{Burger-Wendland} by exploiting the techniques of \cite{Volpert-Hudjaev}, if $a$ is strictly positive on a subinterval and vanishes elsewhere; the limit $\eps\to0$ is considered as well. At last, traveling-wave solution for equation \eqref{eq:sd-sola} have been studied in \cite{GK}; we refer to \cite{Corli-Malaguti, Corli-diRuvo-Malaguti} for applications to collective movements in the case $g=g(u)$. 

\smallskip

The main results of this papers are two. First, for $\eps$ fixed, we prove the well-posedness of  \eqref{eq:sd-sola}-\eqref{e:id} by compactness. Then we consider the limit $\eps\to0$ and show that the solutions of \eqref{eq:sd-sola}-\eqref{e:id} converge to the entropy solution of \eqref{eq:sd-solac}-\eqref{e:id}; in this second part we exploit the recent extension of the theory of $H$-measures to $\R^N$ due to Panov \cite{Panov-err, Panov}. In both cases we aim at providing {\em simple} proofs; in particular we avoid the deep but heavy techniques of \cite{Volpert-Hudjaev} while keeping however a sufficiently high level of generality.

Our results can be extended to the more general equation
\begin{equation*}
\pt u+\dv{f(t,x,u)}=\eps\Delta A(t,x,u) + g(t,x,u)
\end{equation*}
by exploiting the techniques of \cite{CKMR,  Panov-err, Panov} However, we always refer to equation \eqref{eq:sd-sola}.

%%%%%%%%%%%%%%%%%%%%%%%%%%%%%%%%%%%%%%%

\section{Main results}
\label{sec:main}

We fix $T>0$ and denote $R_T:=[0,T]\times \R^N$. We consider the initial-value problem
\begin{equation}
\label{eq:sd}
\begin{cases}
\pt u+\dv{f(u)}=\eps\Delta A(u)+g(t,x,u),&\quad t\in(0,T],\,x\in\R^N,\\
u(0,x)=u_0(x),&\quad x\in\R^N.
\end{cases}
\end{equation}
We denote by $\Lip(E)$ the set of Lipschitz-continuous functions in a set $E$ and make the following assumptions:
\begin{Assumptions}

\item \label{ass:f} $f\in \Lip\left([0,1];\R^N\right)$, $f(0)=0$;

\item \label{ass:a} $A\in \Lip\left([0,1]\right)$, $A'(u)\ge0$ for a.e. $u\in[0,1]$;
 
\item \label{ass:g}  $g(\cdot,x,\cdot)\in \Lip([0,T]\times [0,1])$ for a.e. $x\in\R^N$,  $g(t,\cdot,u)\in (L^1\cap BV)(\R^N)$ for every $(t,u)\in[0,T]\times [0,1]$ and satisfies, for some constant $\kappa>0$,\begin{align}
& g(\cdot,\cdot,1)\le 0\le g(\cdot,\cdot,0),
\label{e:g01}
\\
&\left|g_u(t,x,u)\right|\le \kappa, \quad \hbox{ for a.e. } (t,x,u)\in R_T\times[0,1],
\label{e:gu}
\\
&\sup_{(t,u)\in[0,T]\times[0,1]}
\left\{
\int_{\R^N} \left|g(t,x,u)\right| dx,
\ \left|Dg(t,\cdot, u)\right|(\R^N),
\int_{\R^N} \left|g_t(t,x,u)\right|\, dx\right\}\le \kappa;
\label{e:gdu}
\end{align}

\item \label{ass:u0} $u_0\in L^1(\R^N)$ and $0\le u_0\le 1$ a.e. in $\R^N$.
\end{Assumptions}
Assumption \eqref{e:g01} is only exploited to deduce that solutions are valued in $[0,1]$ a.e. if ({\bf H.4}) holds. In the case $N=1$, an example of source term $g$ satisfying \eqref{e:g01} and modeling either entries or exits in a road is \cite{Bagnerini-Colombo-Corli}
\[
g(t,x,u) = \chi_I(x)h(t,u),
\]
where $\chi_I$ is the characteristic function of some bounded  open interval  $I\subset \R^N$ and $N=1$; the interval $I$ may be replaced by the union of disjoint bounded connected open sets. The function $h\in\Lip\left([0,T]\times [0,1]\right)$ satisfies $h(t,0)>0=h(t,1)$ and $h\ge0$ in the case of an entry, $h(t,0)=0>h(t,1)$ and $h\le0$ in the case of an exit, $h(t,1)<0<h(t,0)$ in the case of an inflow-outflow access.

We denote by $C_c^\infty([0,T)\times\R^N)$ the set of functions $\vfi\in C^\infty([0,T)\times\R^N)$ with compact support.
\begin{definition}
\label{def:sol-sd} 
We say that $u\in  L^\infty_{loc}\left(R_T\right)$ is a weak solution of the initial-value problem \eqref{eq:sd} if for every test function $\vfi\in C_c^\infty([0,T)\times\R^N)$
we have
\begin{equation}
\label{eq:def-weak-sd}
\int_0^\infty\int_{\R^N}\left( u\pt \vfi+f(u)\cdot\nabla \vfi+\eps A(u)\Delta \vfi+g(t,x,u)\vfi\right) dtdx+\int_{\R^N} u_0(x)\vfi(0,x)dx=0.
\end{equation}
Moreover, we say that $u$ is an entropy solution of \eqref{eq:sd} if it is a weak solution of \eqref{eq:sd} and for every nonnegative test function $\vfi\in C_c^\infty([0,T)\times\R^N)$ and every convex entropy $\eta\in C^2(\R)$ we have
\begin{equation}
\label{eq:def-entr-sd}
\begin{split}
\int_0^\infty\int_{\R^N}\left( \eta(u)\pt \vfi+q(u)\cdot\nabla\vfi+\eps \mathcal{A}(u)\Delta \vfi+g(t,x,u)\eta'(u)\vfi\right) dtdx&\\
+\int_{\R^N} \eta\left(u_0(x)\right)\vfi(0,x)dx&\ge 0,
\end{split}
\end{equation}
where the entropy flux $q\in C^2(\R;\R^N)$ and the entropy diffusion $\mathcal{A}\in \Lip(\R)$ satisfy
\begin{equation}
\label{eq:def-q-A-sd}
q'=\eta'f',\qquad \mathcal{A}'=\eta'A'
\end{equation}
 for a.e. $(t,x)\in R_T$.
 \end{definition}

Notice that we can locally approximate the usual Kruzhkov entropies with $C^2$ convex functions in a uniform way.

In this paper we study two problems:
\begin{itemize}

\item first, we keep $\eps>0$ fixed and prove the existence and uniqueness of weak entropy solutions to the initial-value problem \eqref{eq:sd};

\item second, we consider the limit $\eps\to0$ in \eqref{eq:sd} and show the convergence of the solutions $u_\eps$ of \eqref{eq:sd} to the unique entropy solution of the initial-value problem for the corresponding balance law
\begin{equation}
\label{eq:cl}
\begin{cases}
\pt u+\dv{f(u)}=g(t,x,u),&\quad t\in(0,T],\,x\in\R^N,\\
u(0,x)=u_0(x),&\quad x\in\R^N.
\end{cases}
\end{equation}

\end{itemize}
Here follow our main results. Below, we denote by $BV(\R^N)$ the space of functions with bounded variation and by $|Du|$ the total variation of $u\in BV(\R^N)$, see \cite{Ambrosio-Fusco-Pallara}; $Du$ is a finite Radon measure in $\R^N$. The space of finite Radon measures on $\R^N$ ($R_T$) is denoted by $\mathcal{M}(\R^N)$ ($\mathcal{M}(R_T)$, respectively).

\begin{theorem}
\label{th:main-sd}
Assume {\rm\ref{ass:f}}, {\rm\ref{ass:a}}, {\rm\ref{ass:g}}, {\rm\ref{ass:u0}} and 
\begin{equation}
\label{eq:ass-u0}
u_0\in BV(\R^N),\qquad \nabla A(u_0)\in BV(\R^N).
\end{equation}
Then, the Cauchy problem \eqref{eq:sd}  admits a unique entropy solution $u$ in the sense of Definition \ref{def:sol-sd}. The solution has the following properties: 
\begin{align}
\label{eq:th-main-sd1.1}
&0\le u(\cdot,\cdot)\le 1\ a.e.\ in\ R_T,\qquad u\in L^\infty\left(0,T; L^1(\R^N)\right)\cap BV(\R_T),\\
\label{eq:th-main-sd1.2}
&\norm{u(t,\cdot)-u(s,\cdot)}_{L^1(\R^N)}\le L |t-s|,\qquad \hbox{ for a.e. } 0\le t,\,s\le T,
\\
\label{eq:th-main-sd1.3}
&\nabla A(u)
\in L^\infty\left(0,T;  BV(\R^N)\right),
\end{align}
for some constant $L>0$ depending on $u_0$.
Moreover, if $u$ and $v$ are the entropy solutions of \eqref{eq:sd} corresponding to the initial data $u_0$ and $v_0$, we have
\begin{equation}
\label{eq:th-main-sd2}
\norm{u(t,\cdot)-v(t,\cdot)}_{L^1(\R^N)}\le e^{\kappa t} \norm{u_0-v_0}_{L^1(\R^N)},
\end{equation}
for almost every $0\le t\le T$.
\end{theorem}
The proof of Theorem \ref{th:main-sd} is based on a nondegenerate parabolic regularization of equation \eqref{eq:sd}.

\begin{remark} By \eqref{eq:th-main-sd1.3} we deduce that $A(u)$ is locally Lipschitz-continuous with respect to $x$ for a.e. $t\in(0,T]$. If $A$ is strictly monotone, then it is invertible and $A^{-1}$ is continuous; as a consequence, we have that $u=A^{-1}\left(A(u)\right)$ is continuous with respect to $x$ for a.e. $t\in(0,T]$. Moreover, $u(t,\cdot)\in\Lip(\R^N)$ for a.e. $t\in(0,T]$.
\end{remark}

\smallskip

Next, we consider the Cauchy problem \eqref{eq:sd} where the initial data are allowed to depend on $\eps$, namely,
\begin{equation}
\label{eq:sd1}
\begin{cases}
\pt \ue+\dv{f(\ue)}=\eps\Delta A(u_\eps) + g(t,x,\ue),&\quad t\in(0,T],\,x\in\R^N,\\
\ue(0,x)=u_{0,\eps}(x),&\quad x\in\R^N,
\end{cases}
\end{equation}
and investigate the limit $\eps\to0$. To this aim, beside \ref{ass:f}--\ref{ass:u0} we need one further assumption. 
%By denoting with ${\rm meas}$ the Lebesgue measure in $\R$, 
We require:
\begin{Assumptionss}

\item \label{ass:f'} for every  $\xi\in \R^N$ the map $u\in[0,1]\mapsto f(u)\cdot\xi$ is not affine on any nontrivial intervals.

\end{Assumptionss}

In the case $N=1$ condition \ref{ass:f'} encompasses both the case of a convex flux function $f$, which occurs in several model of collective behaviors \cite{Garavello-Piccoli_book, Rosini_book}, as well as the case of a function $f$ with isolated inflection points \cite{Colombo-Rosini2005}. It does not hold, however, in intervals where $f(u) = uv(u)$, where the velocity $v(u)$ is a constant function of $u$; this happens in Dick-Greenberg model of traffic flow \cite{Dick, Greenberg}. 

On the initial datum $u_{0,\eps}$ we assume, for $u_0$ still satisfying assumption \ref{ass:u0},
\begin{align}\label{ass:u-0}
&0\le u_{0,\eps} \le 1\ \text{a.e. in} \ \R^N, \quad \norm{u_{0,\eps}}_{L^2(\R^N)}\le \norm{u_{0}}_{L^2(\R^N)},
\\
\label{ass-a}
 &u_{0,\eps}\to\, u_{0}\text{ a.e. and in $L^p_{loc}(\R^N),\, 1\leq p< \infty$, as $\eps\to0$}.
\end{align}
About \eqref{ass-a}, notice that if $u_{0,\eps}\to u_{0}$ in $L^p_{loc}(\R^N)$ for some $p$ with $1\leq p< \infty$, then $u_{0,\eps}\to u_{0}$ in $L^p_{loc}(\R^N)$ for every $p$ with $1\le p<\infty$ because of \eqref{ass:u-0}.

Solutions to the limit problem \eqref{eq:cl} are meant in the following sense.
\begin{definition}
\label{def:sol}
We say that $u\in  L^{\infty}\left(R_T\right)$ is an entropy solution of the initial-value problem \eqref{eq:cl} if for every test function $\vfi\in C_c^\infty([0,T)\times\R^N)$ we have
\begin{equation}
\label{eq:def-weak-sd1}
\int_0^\infty\!\!\!\int_{\R^N}\left( u\pt \vfi+f(u)\cdot\nabla \vfi+g(t,x,u)\vfi\right) dtdx+\int_{\R^N} u_0(x)\vfi(0,x)dx=0
\end{equation}
and, moreover, for every convex function $\eta\in  C^2(\R)$ and $q$ as in \eqref{eq:def-q-A-sd}, the inequality
\begin{equation}
\label{eq:OHentropy}
\int_0^\infty\!\!\!\int_{\R^N}\left( \eta(u)\pt \vfi + q(u)\cdot\nabla \vfi+g(t,x,u)\eta'(u)\vfi\right) dtdx
+\int_{\R^N} \eta\left(u_0(x)\right)\vfi(0,x)dx\ge 0
%\pt \eta(u)+ \px q(u)-\eta'(u)g(t,x,u) \le 0,% \qquad     q(u)=\int^u f'(\xi) \eta'(\xi)\, d\xi,
\end{equation}
holds for every nonnegative test function $\vfi\in C_c^\infty([0,T)\times\R^N)$. 
\end{definition}

Due to the boundedness of the solutions, our definition of entropy solution is equivalent to that of Kruzhkov \cite{Kruzhkov}. Indeed, we can always uniformly approximate the Kruzhkov entropies with $C^2$ convex functions. As a consequence, due to the uniqueness result proved in \cite{Kruzhkov}, entropy solutions in the sense of Definition \ref{def:sol} for \eqref{eq:cl} are unique.

Following \cite{CKMR, Panov-err, Panov}, we prove the following result.

\begin{theorem}
\label{th:main}
Under hypotheses {\rm\ref{ass:f}}, {\rm\ref{ass:a}}, {\rm\ref{ass:g}}, {\rm\ref{ass:u0}}, assume moreover {\rm\ref{ass:f'}}, \eqref{ass:u-0} and \eqref{ass-a}. Let $u_\eps$ be the corresponding solution of \eqref{eq:sd1} provided by Theorem \ref{th:main-sd}.
Then we have
\begin{equation}
\label{eq:convu-1}
u_\eps\to u\qquad \text{a.e. and in $L^p_{loc}\left(R_T\right),\,1\le p<\infty$, as $\eps\to0$},
\end{equation}
where $u$ is the unique entropy solution of \eqref{eq:cl} in the sense of Definition \ref{def:sol}.
\end{theorem}

%%%%%%%%%%%%%%%%%%%%%%%%%%%%%%%%%%%%%%%%%%%%%%
\section{Existence and uniqueness of solutions to \eqref{eq:sd}}
\label{sec:singdiff}

In this section we prove Theorem \ref{th:main-sd}. Our existence argument is based on the following regularization of problem \eqref{eq:sd}. Let $\mu>0$ be given; we consider families of functions $\{f_\mu\}$, $\{A_\mu\}$, $\{g_\mu\}$ in $C^\infty$ that approximate $f$, $A$, $g$, respectively, in the sense
\begin{align}
&f_\mu\to f,\ A_\mu\to A\ \text{uniformly in}\ [0,1],
\label{e:gmuconv0}
\\
&\sup_{u\in[0,1]}\int\!\!\!\int_{R_T}\left|g_\mu(t,x,u)-g(t,x,u)\right|dtdx\to0,
\label{e:gmuconv}
\end{align}
for $\mu\to0+$. Moreover, $f_\mu$ and $A_\mu$ satisfy \ref{ass:f} and \ref{ass:a}, respectively, where $f_\mu$ and $A_\mu$ replace $f$ and $A$; we call $\ref{ass:f}_\mu$ and $\ref{ass:a}_\mu$ such assumptions. Moreover,  
\begin{equation}\label{e:fextramu}
\|f_\mu'\|_{L^\infty([0,1])}\le \norm{f'}_{L^\infty([0,1])}.
\end{equation}
On the other hand, the functions $g_\mu$ satisfy the same assumptions \eqref{e:g01}--\eqref{e:gdu} of $g$, for $g_\mu$ replacing $g$; for sake of brevity we do not write again these assumptions for $g_\mu$ and in the following we shall quote them by $\eqref{e:g01}_\mu$--$\eqref{e:gdu}_\mu$. At last, we consider a family of functions $\{u_{0,\mu}\}\subset C_c^\infty(\R^N)$ satisfying
\begin{align}
\label{eq:ass-mu1}
&0\le u_{0,\mu}\le 1,
\\
\label{eq:ass-mu2}
&u_{0,\mu}\to u_0\quad \text{in $L^p(\R^N),\,1\le p<\infty$},\text{ for }\mu\to0,
\\
\label{eq:ass-mu345}
&\norm{u_{0,\mu}}_{L^1(\R^N)}\le \norm{u_{0}}_{L^1(\R^N)},\quad \norm{u_{0,\mu}}_{L^2(\R^N)}\le \norm{u_{0}}_{L^2(\R^N)},
\quad \norm{u_{0,\mu}'}_{L^1(\R^N)}\le |D(u_0)|(\R^N),
\\
\label{eq:ass-mu67}
&\norm{A_\mu(u_{0,\mu})''}_{L^1(\R^N)}\le \left|D(A(u_0)')\right|(\R^N),\quad \mu\norm{u_{0,\mu}''}_{L^1(\R^N)}\le \kappa_0,
\end{align}
for some constant $\kappa_0>0$ independent from $\eps$ and $\mu$.

Then the non-degenerate parabolic initial-value problem
\begin{equation}
\label{eq:sd-mu}
\begin{cases}
\pt u_\mu+\dv{f_\mu(u_\mu)} = \eps\Delta\left(A_\mu(u_\mu) + \mu u_\mu\right) + g_\mu(t,x, u_\mu),&\quad t>0,\,x\in\R^N,\\
u_\mu(0,x)=u_{0,\mu}(x),&\quad x\in\R^N,
\end{cases}
\end{equation}
has a unique solution $u_\mu\in C^\infty(R_T)\cap W^{1,1} (R_T)$. 
%ahi
\begin{lemma}[{\bf $L^\infty$, $L^1$ and $L^2$ estimates}]
\label{lm:linfty-sd}
For the solution $u_\mu$ of problem \eqref{eq:sd-mu} we have, for $t\in[0,T)$, 
\begin{align}
\label{eq:linfty-sd}
&0\le u_\mu\le 1,
\\
\label{eq:l1-sd}
&\norm{u_\mu(t,\cdot)}_{L^1(\R^N)}\le \norm{u_0}_{L^1(\R^N)}+\kappa t,\\
&\label{eq:l2-sd}
\norm{u_\mu(t,\cdot)}_{L^2(\R^N)}^2+2\eps\int_0^t \norm{\sqrt{a_\mu\left(u_\mu(s,\cdot)\right)+\mu}\ \nabla u_\mu(s,\cdot)}_{L^2(\R^N)}^2ds
\le\norm{u_0}_{L^2(\R^N)}^2+2\kappa t.
\end{align}
\end{lemma}

\begin{proof}
First, the constant functions $u=0$ and $u=1$ are subsolution and supersolution, respectively, of \eqref{eq:sd-mu} by $\eqref{e:g01}_\mu$. Then \eqref{eq:linfty-sd} follows by $\eqref{eq:ass-mu1}$ and the Comparison Principle for (nondegenerate) parabolic equations.

Second, by \eqref{eq:linfty-sd} we have
\begin{align*}
\frac{d}{dt}\int_{\R^N} |u_\mu|dx=&\frac{d}{dt}\int_{\R^N} u_\mu dx=\int_{\R^N} \pt u_\mu dx
\\
=&-\int_{\R^N} \dv{f_\mu(u_\mu)} dx+\eps\int_{\R^N} \dv{\left(a_\mu(u_\mu) + \mu\right)\nabla u_\mu} dx + \int_{\R^N} g_\mu(t,x,u_\mu)dx\le \kappa
\end{align*}
by $\eqref{e:gdu}_{\mu,1}$, because the first two terms in the last line vanish. An integration over $(0,t)$ and $\eqref{eq:ass-mu345}_1$ prove \eqref{eq:l1-sd}.

To prove \eqref{eq:l2-sd} we denote $h_\mu(u) = \int^u f_\mu'(s)s\, ds$. Then, by \eqref{eq:linfty-sd}, $\eqref{e:gdu}_{\mu,1}$ and $\ref{ass:a}_\mu$ we have
\begin{align*}
\frac{d}{dt}\int_{\R^N}& \frac{u_\mu^2}{2}dx=\int_{\R^N} u_\mu \pt u_\mu dx
\\
=&-\int_{\R^N} \dv{f_\mu(u_\mu)} u_\mu dx+\eps\int_{\R^N} \dv{(a_\mu(u_\mu)+\mu)\nabla u_\mu} u_\mu dx +\int_{\R^N} g_\mu(t,x,u_\mu)u_\mu dx
\\
=&\underbrace{-\int_{\R^N} \dv{h_\mu(u_\mu)} dx}_{=0}
-\eps\int_{\R^N} \left(a_\mu(u_\mu)+\mu\right)|\nabla u_\mu|^2 dx +\int_{\R^N} g_\mu(t,x,u_\mu)u_\mu dx
\\
\le&-\eps\int_{\R^N} \left(a_\mu(u_\mu)+\mu\right)|\nabla u_\mu|^2 dx +\kappa.
\end{align*}
An integration over $(0,t)$ and $\eqref{eq:ass-mu345}_2$ prove \eqref{eq:l2-sd}.
\end{proof}
%%%

By multiplying equation $\eqref{eq:sd-mu}_1$ by $\eta'(u_\mu)$, with If $\eta\in C^2(\R)$, and then integrating by parts we deduce
\begin{align*}
&\eps\int\!\!\!\int_{R_T}\eta''(u_\mu)\left(a_\mu(u_\mu) + \mu\right)\left|\nabla u_\mu\right|^2dtdx 
\\
&= 
\int\!\!\!\int_{R_T}\eta'(u_\mu)g_\mu(t,x,u_\mu)\,dtdx 
- \int_{\R^N}\eta\left(u_\mu(T,x)\right)\, dx + \int_{\R^N}\eta\left(u_\mu(0,x)\right)\, dx.
\end{align*}
%\footnote{C'era anche questo che mi pare non serva (ed e' sbagliato: $\eta$ non si annulla in $0$): 
%
%Moreover, by \eqref{eq:linfty-sd} we deduce $\left| \eta(u_\mu)\right|\le C_1\left| u_\mu\right|$, for $C_j = \max_{\xi\in[0,1]}|\eta^{(j)}(\xi)|$, $j=1,2$. By \eqref{eq:l1-sd} it follows
%\[
%\left|\int_{\R^N}\eta\left(u_\mu(t,x)\right)\, dx\right|\le C_1\norm{u_0}_{L^1(\R^N)}+\kappa T,
%\]
%for $t\in[0,T]$. Analogously,
%\[
%\left|\int\!\!\!\int_{R_T}\eta'(u_\mu)g_\mu(t,x,u_\mu)\,dtdx\right| \le C_1 \int\!\!\!\int_{R_T}|g_\mu(t,x,u_\mu)|\,dtdx\le C_1\kappa T,
%\]
%by $\eqref{e:gdu}_{\mu,1}$.} 
If we denote
\[
\nu_\mu := \eps\eta''(u_\mu)\left(a_\mu(u_\mu) + \mu\right)\left|\nabla u_\mu\right|^2,
\]
then by \eqref{eq:l2-sd} we deduce 
\[
\int\!\!\!\int_{R_T}|\nu_\mu| dtdx\le C_2\left(\norm{u_0}_{L^2(\R^N)}^2+2\kappa T\right)=:\bar C,
\]
where $C_2 = \max_{\xi\in[0,1]}|\eta''(\xi)|$. As a consequence, unless of extracting a subsequence, we have that $\nu_\mu\to \nu$ for some $\nu\in \mathcal{M}(R_T)$. Moreover, we have $\nu\ge0$ by $\ref{ass:a}_\mu$ if $\eta$ is convex. Notice that $\nu_\mu$ also depend on $\eps$ but it is uniformly bounded also with respect to $\eps$. As a consequence, the limit measure, that we could write $\nu=\nu_\eps$, satisfies
\begin{equation}\label{e:nuepsbdd}
0\le\nu_\eps(R_T)\le \bar C,
\end{equation}
where the constant $\bar C$ does not depend on $\eps$.

%%%
\begin{lemma}[{\bf $BV$ estimates}]
\label{lm:BVx-sd}
For every $t\in[0,T)$ and $i\in\{1,..,N\}$ we have
\begin{align}
\label{eq:BVx-sd}
\norm{\p_{x_i} u_\mu(t,\cdot)}_{L^1(\R^N)} & \le |D(u_0)|(\R^N)\, e^{\kappa t}+e^{\kappa t}-1,
\\
\label{eq:BVt-sd}
\norm{\pt u_\mu(t,\cdot)}_{L^1(\R^N)} & \le \left(\norm{f'}_{L^\infty(0,1)}|D(u_0)|(\R^N) +\eps \kappa_0+\eps \left|D(A(u_{0})')\right|(\R^N)+\kappa\right)e^{\kappa t}+e^{\kappa t}-1.
\end{align}
\end{lemma}

\begin{proof}
We first prove \eqref{eq:BVx-sd}. By differentiating the equation in \eqref{eq:sd-mu} with respect to $x_i$ we get
\begin{align*}
\pt\p_{x_i} u_\mu&+\dv{f_\mu'(u_\mu)\p_{x_i} u_\mu}\\
=&\eps\dv{(a_\mu(u_\mu)+\mu)\nabla\p_{x_i} u_\mu}+\eps\dv{a_\mu'(u_\mu)\p_{x_i} u_\mu\nabla u_\mu}+g_{\mu,x}(t,x,u_\mu)+g_{\mu,u}(t,x,u_\mu)\p_{x_i} u_\mu.
\end{align*}
We have
\begin{align*}
\frac{d}{dt}\int_{\R^N}& |\p_{x_i} u_\mu|dx=\int_{\R^N} \pt\p_{x_i} u_\mu \sgn{\p_{x_i} u_\mu}dx
\\
=&-\int_{\R^N} \dv{ f_\mu'(u_\mu)\p_{x_i} u_\mu} \sgn{\p_{x_i} u_\mu}dx+\eps\int_{\R^N} \dv{(a_\mu(u_\mu)+\mu)\nabla\p_{x_i} u_\mu} \sgn{\p_{x_i} u_\mu}dx\\
&+\eps\int_{\R^N} \dv{a_\mu'(u_\mu)\p_{x_i} u_\mu\nabla u_\mu} \sgn{\p_{x_i} u_\mu}dx
\\
&+\int_{\R^N} g_{\mu,x_i}(t,x,u_\mu)\,\sgn{\p_{x_i} u_\mu}dx+\int_{\R^N} g_{\mu,u}(t,x,u_\mu)|\p_{x_i} u_\mu| dx
\\
=&\underbrace{\int_{\R^N} f_\mu'(u_\mu)\p_{x_i} u_\mu\nabla\p_{x_i} u_\mu \delta_{\{\p_{x_i} u_\mu=0\}}dx}_{=0}\,
\underbrace{-\eps\int_{\R^N} \left(a_\mu(u_\mu)+\mu\right)|\nabla\p_{x_i} u_\mu|^2 \delta_{\{\p_{x_i} u_\mu=0\}}dx}_{\le 0}
\\
&\underbrace{-\eps\int_{\R^N} a_\mu'(u_\mu)\p_{x _i}u_\mu\nabla u_\mu\cdot\nabla\p_{x_i} u_\mu \delta_{\{\p_{x_i} u_\mu=0\}}dx}_{=0}
\\
&+\int_{\R^N} g_{\mu,x_i}(t,x,u_\mu)\sgn{\p_{x_i} u_\mu}dx+\int_{\R^N} g_{\mu,u}(t,x,u_\mu)|\p_{x_i} u_\mu| dx
\\
\le &\ \kappa+\kappa\int_{\R^N} |\p_{x_i} u_\mu| dx,
\end{align*}
by $\ref{ass:a}_\mu$, $\eqref{e:gdu}_{\mu,2}$ and $\eqref{e:gu}_\mu$. The Gronwall Lemma  and $\eqref{eq:ass-mu345}_3$ prove \eqref{eq:BVx-sd}.

The proof of \eqref{eq:BVt-sd} is analogous. Differentiating the equation in \eqref{eq:sd-mu} with respect to $t$ and proceeding as above we easily deduce, by exploiting $\eqref{e:gdu}_{\mu,3}$ instead of $\eqref{e:gdu}_{\mu,2}$,
\[
\frac{d}{dt}\int_{\R^N} |\pt u_\mu|dx \le \kappa+\kappa\int_{\R^N} |\pt u_\mu| dx.
\]
%we get
%\begin{align*}
%\ptt u_\mu+\px\left( f'(u_\mu)\pt u_\mu\right)=&\eps\px\left((a(u_\mu)+\mu)\ptx u_\mu\right) + \eps\px\left(a'(u_\mu)\pt u_\mu\px u_\mu\right)\\
%&+g_t(t,x,u_\mu)+g_u(t,x,u_\mu)\pt u_\mu.
%\end{align*}
%We have
%\begin{align*}
%\frac{d}{dt}\int_\R |\pt u_\mu|dx=&\frac{d}{dt}\int_\R \ptt u_\mu \sgn{\pt u_\mu}dx\\
%=&-\int_\R \px\left( f'(u_\mu)\pt u_\mu\right) \sgn{\pt u_\mu}dx+\eps\int_\R \px\left((a(u_\mu)+\mu)\ptx u_\mu\right) \sgn{\pt u_\mu}dx\\
%&+\eps\int_\R \px\left(a'(u_\mu)\pt u_\mu\px u_\mu\right) \sgn{\pt u_\mu}dx\\
%&+\int_\R g_t(t,x,u_\mu)\sgn{\pt u_\mu}dx+\int_\R g_u(t,x,u_\mu)|\pt u_\mu| dx\\
%=&\underbrace{\int_\R f'(u_\mu)\pt u_\mu\ptx u_\mu \delta_{\{\pt u_\mu=0\}}dx}_{=0}
%\underbrace{-\eps\int_\R \left(a(u_\mu)+\mu\right)(\ptx u_\mu)^2 \delta_{\{\pt u_\mu=0\}}dx}_{\le 0}\\
%&\underbrace{-\eps\int_\R a'(u_\mu)\pt u_\mu\px u_\mu\ptx u_\mu \delta_{\{\pt u_\mu=0\}}dx}_{=0}\\
%&+\int_\R g_t(t,x,u_\mu)\sgn{\pt u_\mu}dx+\int_\R g_u(t,x,u_\mu)|\pt u_\mu| dx\\
%\le &\kappa+\kappa\int_\R |\pt u_\mu| dx.
%\end{align*}
Thanks to Gronwall Lemma, the equation in \eqref{eq:sd-mu}, \eqref{eq:ass-mu1}--\eqref{eq:ass-mu67},  \eqref{e:fextramu} and \eqref{eq:linfty-sd} we have 
\begin{align*}
&\norm{\pt u_\mu(t,\cdot)}_{L^1(\R^N)}\le \norm{\pt u_\mu(0,\cdot)}_{L^1(\R^N)} e^{\kappa t}+e^{\kappa t}-1
\\
&=\norm{-f_\mu'(u_{0,\mu})u_{0,\mu}'+\eps\mu u_{0,\mu}''+\eps A_\mu(u_{0,\mu})'' + g_\mu(0,\cdot, u_{0,\mu})}_{L^1(\R^N)}e^{\kappa t}+e^{\kappa t}-1
\\
&\le\left(\norm{f_\mu'}_{L^\infty(0,1)}\norm{u_{0,\mu}'}_{L^1(\R^N)}+\eps\mu \norm{u_{0,\mu}''}_{L^1(\R^N)}+\eps \norm{A_\mu(u_{0,\mu})''}_{L^1(\R^N)}
+\norm{g_\mu(0,\cdot,u_{0,\mu})}_{L^1(\R^N)}\right)e^{\kappa t}
\\
&\quad +e^{\kappa t}-1
\\
&\le \left(\norm{f'}_{L^\infty(0,1)}|D(u_0)|(\R^N)+\eps \kappa_0+\eps \left|D(A(u_{0})')\right|(\R^N) +\kappa\right)e^{\kappa t}+e^{\kappa t}-1,
\end{align*}
by $\eqref{eq:ass-mu345}_3$, $\eqref{eq:ass-mu67}_2$, $\eqref{eq:ass-mu67}_1$ and $\eqref{e:gdu}_1$. This proves \eqref{eq:BVt-sd}.
\end{proof}

\begin{proof}[Proof of Theorem \ref{th:main-sd}]
By \eqref{eq:l1-sd}, \eqref{eq:BVx-sd}, \eqref{eq:BVt-sd} and \cite[Theorem 3.2.3]{Ambrosio-Fusco-Pallara} the family $\{u_\mu\}$ converges (unless of choosing a subsequence) to some function $u$. The function $u$ solves the equation in \eqref{eq:sd} by the Dominated Convergence Theorem: this follows by \eqref{e:gmuconv0}, \eqref{e:gmuconv} and by noticing that for any compact set $K\subset R_T$ we have, by $\eqref{e:gu}_\mu$ and \eqref{e:gmuconv},
\begin{align*}
&\int\!\!\!\int_K\left|g_\mu(t,x,u_\mu)-g(t,x,u)\right| dtdx 
\\ 
&\le \int\!\!\!\int_K\left|g_\mu(t,x,u_\mu)-g_\mu(t,x,u)\right| dtdx
+\int\!\!\!\int_K\left|g_\mu(t,x,u)-g(t,x,u)\right| dtdx
\\
& \le \kappa \int\!\!\!\int_K\left|u_\mu-u \right| dtdx +\sup_{\xi\in[0,1]}\int\!\!\!\int_K\left|g_\mu(t,x,\xi)-g(t,x,\xi)\right| dtdx.
\end{align*}
Estimate \eqref{eq:th-main-sd1.1} trivially follows from \eqref{eq:linfty-sd}. The initial data are assumed in the limit sense because of \eqref{eq:ass-mu2}.

Estimate \eqref{eq:th-main-sd1.2} is obtained by writing it for $u_\mu$, applying \eqref{eq:BVt-sd} and then passing to the limit using again the Dominated Convergence Theorem; more precisely we find 
\[
L=\left(\norm{f'}_{L^\infty(0,1)}|D(u_0)|(\R^N) +\eps \kappa_0+\eps \left|D(A(u_{0})')\right|(\R^N)+\kappa\right)e^{\kappa T}+e^{\kappa T}-1. 
\]
To prove that  $u\in L^\infty(0,T; BV(\R^N))$ we observe that $u_{\mu,t}\in L^\infty([0,T); L^1(\R^N))$ by \eqref{eq:BVt-sd}, hence $u_{\mu,t}\in L^p([0,T); L^1(\R^N))$ for every $p\in[1,\infty]$. Then $u_t\in L^\infty([0,T); \mathcal{M}(\R^N))$, hence $u\in BV(R_T)$.

Property \eqref{eq:th-main-sd1.3} follows by writing equation $\eqref{eq:sd}_1$ as 
\[
\eps\dv{\nabla A(u)} = g(t,x,u)-\pt u-\dv{f(u)}.
\] 
The first summand on the right-hand side belongs to $L^\infty([0,T); L^1(\R^N))$ by \ref{ass:g}, the second and the third ones to $L^\infty([0,T); \mathcal{M}(\R^N))$ because of \eqref{eq:BVx-sd} and \eqref{eq:BVt-sd}. Then $\px A(u)\in L^\infty([0,T); BV(\R^N))$, i.e. \eqref{eq:th-main-sd1.3}.

At last, we have to prove \eqref{eq:th-main-sd2}. Let $u$ and $v$ be the entropy solutions of \eqref{eq:sd} corresponding to the initial data $u_0$ and $v_0$ satisfying \ref{ass:u0} and \eqref{eq:ass-u0}.
Using the doubling of variables \cite{Kruzhkov} we can prove
\begin{align*}
\pt |u-v|&+\dv{\sgn{u-v}(f(u)-f(v))}\\
\le& \eps\dv{\sgn{u-v}\nabla(A(u)-A(v))}+|g(t,x,u)-g(t,x,v)|\\
\le& \eps\dv{\sgn{u-v}\nabla(A(u)-A(v))}+\kappa |u-v|,
\end{align*}
by \eqref{e:gu} and then \eqref{eq:th-main-sd2}, whence uniqueness follows.
\end{proof}

%%%%%%%%%%%%%%%%%%%%%%%%%%%%%%%%%%%%%%%%%%%%%%%
\section{The vanishing diffusion limit}\label{sec:singlim}

In this section we prove Theorem \ref{th:main}. The following result is proved exactly as Lemma \ref{lm:linfty-sd}; indeed, the further regularity assumptions required there on $u_{0,\mu}$ are not needed. We rewrite it below for reference in this section. 

\begin{lemma}[{\bf $L^\infty$ and $L^2$ estimates}]
\label{lm:linfty-sd1}
Let $\ue$ be the entropy solution to problem \eqref{eq:sd1}. Then 
\begin{align}
\label{eq:linfty-sd-1}
&0\le \ue\le 1,\ a.e.\ in\ R_T,
\\
\label{eq:l2-sd1}
&\norm{\ue(t,\cdot)}_{L^2(\R^N)}^2+2\eps\int_0^t\!\!\!\int_{\R^N} a(\ue)|\nabla\ue|^2dsdx\le\norm{u_0}_{L^2(\R^N)}^2+2\kappa t, \quad t\in(0,T].
\end{align}
\end{lemma}

We can now pass to the limit in a subsequence of the family $\{u_{\eps}\}_{\eps>0}$ of diffusive approximations \eqref{eq:sd1} by using the compactness result for multidimensional scalar conservation laws proved in \cite{Panov, Panov-err}.
The two key results on which the proof of Theorem \ref{th:main} is based are the following.

\begin{theorem}[\cite{Panov}] 
\label{th:comp} 
Let  $\{v_\nu\}_{\nu>0}$ be a family of functions defined on $R_T$. If $\{v_\nu\}_{\nu>0}$ lies in a bounded set of $L^\infty(R_T)$ and for every 
convex $C^2$ entropy function $\eta:\R\to\R$ with corresponding entropy flux $q:\R\to\R^N$  the family
$$
\left\{\pt\eta(v_\nu)+\dv{q(v_\nu)}\right\}_{\nu>0} 
$$  
lies in a compact set of $H^{-1}_{loc}(R_T)$, then there exist a sequence $\{\nu_n\}_{n\in\N}\subset(0,\infty),\,\nu_n\to 0,$ and a map $v\in L^\infty(R_T)$ such that
$$
v_{\nu_n}\to v\quad \text{a.e. and in }  L^p_{loc}(R_T),\,1\le p<\infty.
$$
\end{theorem}

\begin{theorem}[\cite{Murat:Hneg}]
\label{th:Murat} 
Let $\Omega$ be a open subset of $\R^N$, $N\ge 2$.  Suppose the sequence $\left\{\CL_n\right\}_{n\in \N}$ of distributions is bounded in $W^{-1,\infty}(\Omega)$.  Suppose also that
$$
\CL_n=\CL_{1,n} + \CL_{2,n},
$$
where $\left\{\CL_{1,n}\right\}_{n\in \N}$ lies in a compact subset   of $\Hneg(\Omega)$ and $\left\{\CL_{2,n}\right\}_{n\in \N}$ lies in a bounded subset of ${\mathcal M}_{loc}(\Omega)$.  Then
$\left\{\CL_n\right\}_{n\in \N}$ lies in a compact subset of $H^{-1}_{loc}(\Omega)$.
\end{theorem}

We are now ready for the proof of Theorem \ref{th:main}.

\begin{proof}[Proof of Theorem \ref{th:main}]
Let $\eta:\R\to\R$ be any convex $C^2$ entropy function and $q:\R\to\R^N$ be the corresponding entropy flux. We claim that 
\begin{equation}\label{eq:LLL}
    \pt  \eta(\ue)+\dv{q(\ue)}
    =\underbrace{\eps\Delta\mathcal{A}(\ue)}_{=:\CL_{1,\eps}}
    \, \underbrace{-\nu_\eps}_{=: \CL_{2,\eps}}
     \, \underbrace{+\eta'(\ue)g(t,x,\ue) }_{=: \CL_{3,\eps}},
\end{equation}
where  $\CL_{1,\eps}$, $\CL_{2,\eps}$, $\CL_{3,\eps}$ are distributions. To prove \eqref{eq:LLL} we come back to the approximation used in \eqref{eq:sd-mu}; we now emphasize the dependence on $\eps$ and denote $u_{\mu,\eps} := u_\mu$. We multiply the equation in \eqref{eq:sd-mu} by
$\eta'(u_{\mu,\eps})$; then, identity 
\begin{align}\nonumber
    &\pt  \eta(\ume)+\dv{q(\ume)}
    \\
    &=\eps\dv{\eta'(\ume)\nabla A_\mu(\ume)}
    \, -\eps \eta''(\ume)\nabla \ume \cdot\nabla A_\mu(\ume)
     \, +\eta'(\ume)g_\mu(t,x,\ume),
\label{eq:LLLe}
\end{align}
surely holds in the strong sense for $u_{\mu,\eps}$, where now $\mathcal{L}_{j,\eps}$ are smooth functions for $j=1,2,3$ and $A=A_\mu$, $g=g_\mu$. Both terms $\eta(u_{\mu,\eps})$ and 
$q(u_{\mu,\eps})$ belong to $L_{loc}^1(R_T)$ and $L_{loc}^1(R_T;\R^N)$ uniformly with respect to $\eps$ because of \eqref{eq:linfty-sd}. The same holds for the last term $\eta'(u_{\mu,\eps})g_\mu(t,x,u_{\mu,\eps})$ by $\eqref{e:gdu}_{\mu,1}$. Moreover, $\eps\nu_\eps$ belongs to $L^1$ uniformly with respect to $\eps$, and $\mathcal{A}_\mu(u_{\mu,\eps})$ belongs to $L_{loc}^1(R_T)$ uniformly with respect to $\eps$. Then, by Theorem \ref{th:comp} we can pass to the limit for $\mu\to0$. We already know that $u_{\mu,\eps}\to u_{\eps}$ a.e.; then $\eta(u_{\mu,\eps})$ and $q(u_{\mu,\eps})$ weakly converge, $\eps\mathcal{A}(u_{\mu,\eps})\to\eps\mathcal{A}(u_\eps)$ and then $\eps\Delta\mathcal{A_\mu}(u_{\mu,\eps})\to \eps\Delta\mathcal{A}(u_{\eps})$ in $\mathcal{D}'(R_T)$. We deduce that $u_{\eps}$ satisfies equation \eqref{eq:LLL} in the weak sense. This proves the claim.

Next, we claim that for each $t>0$ we have
\begin{enumerate}[{\em (i)}]

\item $\CL_{1,\eps}\to 0$ in $H^{-1}(R_T)$;

\item $\{\CL_{2,\eps}\}_{\eps>0}$ is uniformly bounded in $\mathcal{M}(R_T)$;

\item $\{\CL_{3,\eps}\}_{\eps>0}$ is uniformly bounded in $L^1\left(R_T\right)$.

\end{enumerate}
About {\em (i)}, by \eqref{eq:linfty-sd-1} and \eqref{eq:l2-sd1} we deduce
\begin{align*}
\norm{\eps\nabla \mathcal{A}(\ue)}^2_{L^2(R_T;\R^N)}&\le\eps ^2\norm{\eta'}^2_{L^{\infty}(0,1)}\int_{0}^{t}\!\!\!\int_{\R^N} |\nabla A(\ue)|^2 dsdx
\\
&\le\eps^2\norm{\eta'}^2_{L^{\infty}(0,1)}\norm{a}_{L^{\infty}(0,1)}\int_{0}^{t}\!\!\!\int_{\R^N} a(\ue)|\nabla\ue|^2 dsdx\\
&\le\eps\norm{\eta'}^2_{L^{\infty}(0,1)}\norm{a}_{L^{\infty}(0,1)}\left(\norm{u_0}_{L^2(\R^N)}^2+ \kappa t\right)\to0.
\end{align*}
Since $\eps\nabla \mathcal{A}(\ue)\to 0$ in $L^2(R_T;\R^N)$, then $\eps\Delta \mathcal{A}(\ue)\to 0$ in $H^{-1}(R_T)$ and this proves {\em (i)}. Item {\em (ii)} follows by \eqref{e:nuepsbdd}.
At last, by  $\eqref{e:gdu}_1$ and \eqref{eq:linfty-sd-1} we have
\begin{align*}
\norm{\eta'(\ue)g(t,x,\ue)}_{L^1(R_T)}
&\le \norm{\eta'}_{L^{\infty}(0,1)}\int_{0}^{T}\!\!\!\int_{\R^N} |g(t,x,\ue)| dtdx\le \norm{\eta'}_{L^{\infty}(0,1)}\kappa T,
\end{align*}
which concludes the proof of the claim. Therefore, Theorem \ref{th:Murat} implies
\begin{equation}
\label{eq:GMC1-1}
    \text{$\left\{  \pt  \eta(\ue)+\dv{q(\ue)}\right\}_{\eps>0}$
    lies in a compact subset of $\Hneg\left(R_T\right)$.}\end{equation}
By \ref{ass:f'}, \eqref{eq:linfty-sd-1} and the inclusion \eqref{eq:GMC1-1}, the Theorem \ref{th:comp} implies the existence of a subsequence
$\{\uek\}_{k\in\N}\subset \{\ue\}_{\eps>0}$ that converges to a limit function $  u\in L^{\infty}(R_T)$ a.e. and in $L^p_{loc}(R_T)$, for every $1\le p<\infty$, as $k\to\infty$. This proves \eqref{eq:convu-1}.

Now, we prove that $u$ is a distributional solution to \eqref{eq:cl}, i.e., it satisfies \eqref{eq:def-weak-sd1}. We denote for short $u_k:=u_{\eps_{k}}$. Since $u_k$ satisfies \eqref{eq:sd1}, by \eqref{eq:def-weak-sd} we have
\begin{align*}&\int_{0}^{\infty}\!\!\!\int_{\R^N}\left(u_k \pt\phi + f(u_k)\cdot\nabla\phi\right)\,dtdx +\int_{\R^N}u_{k,0}(x)\phi(0,x)\, dx \\
&\qquad\quad+ \int_{0}^{\infty}\!\!\!\int_{\R^N}g(t,x,u_k)\,dtdx +\eps_k\int_{0}^{\infty}\!\!\!\int_{\R^N} A(u_k) \Delta\phi\, dtdx=0.
\end{align*}
Identity \eqref{eq:def-weak-sd1} follows by the Dominated Convergence Theorem because of \eqref{eq:convu-1}, \eqref{eq:linfty-sd-1} and \eqref{eq:l2-sd1}.
%\footnote{To prove that $\px\mathcal{A}(\ue)\in L^2$ and $\px\mathcal{A}(\ue)= \mathcal{A}'(\ue)\px \ue$, argue by approximation: both formulas hold for $u_{\mu,\eps}$ and then pass to the limit.}
We conclude by showing that $u$ satisfies inequality \eqref{eq:OHentropy}. Arguing as in \eqref{eq:LLL}, we obtain, in the weak sense,
\begin{align*}
\pt \eta(u_k) +\dv{q(u_k)} - \eta'(u_k)g(t,x,u_k)
=& \eps \Delta\mathcal{A}(u_k) - \eps\nu_\eps
\le \eps \Delta\mathcal{A}(u_k),
\end{align*}
because of \ref{ass:a} and \eqref{e:nuepsbdd}. Then, for every nonnegative function $\phi\in C_c^\infty(R_T)$ we have
\begin{align*}
\int\!\!\!\int_{R_T}\left( \eta(u_k)\pt\phi + q(u_k)\cdot\nabla\phi\right) dtdx &+ \int\!\!\!\int_{R_T}\eta'(u_k)g(t,x,u_k)\phi\, dtdx \\
&+ \int_{\R^N}\eta(u_{0,k})\phi(0,x)\, dtdx \ge \eps \int\!\!\!\int_{R_T}\mathcal{A}(u_k)\Delta\phi\, dtdx.
\end{align*}
Therefore, \eqref{eq:OHentropy} follows by the Dominated Convergence Theorem because of \eqref{ass-a}, \eqref{eq:convu-1}, \eqref{eq:linfty-sd-1} and \eqref{eq:l2-sd1}. Notice that since entropy solutions to \eqref{eq:cl} are unique, then the whole sequence converges, indeed. 
\end{proof}

%%%%%%%%%%%%%%%%%%%%%%%%%%%%%%%%%%%%%%%%%%%%%%%%
\def\cprime{$'$} \def\cydot{\leavevmode\raise.4ex\hbox{.}}

{\small
\bibliographystyle{abbrv}

\def\cprime{$'$} \def\cydot{\leavevmode\raise.4ex\hbox{.}}

}

\end{document}